\documentclass[a4paper, 12pt]{amsart}
\usepackage{amsmath,latexsym,amssymb, hyperref}
\usepackage{amsrefs}

\makeatletter
\@namedef{subjclassname@2010}{%
  \textup{2010} Mathematics Subject Classification}
\makeatother

\newtheorem{lemma}{Lemma}[section]
\newtheorem{theorem}[lemma]{Theorem}
\newtheorem{corollary}[lemma]{Corollary}
\newtheorem{proposition}[lemma]{Proposition}

\theoremstyle{definition}
\newtheorem{definition}[lemma]{Definition}

\newtheorem{remarks}[lemma]{Remarks}

\newtheorem{example}[lemma]{Example}

\def\N{\mathbb{N}}
\def\C{\mathcal{C}}

\def\Q{\mathbb{Q}}

\def\Aut{{\rm Aut}}

\def\c{\bar{c}}

\def\a{\bar{a}}

\def\acl{{\rm acl}}

\def\Sym{{\rm Sym}}

\def\A{\mathcal{A}}
\def\vphi{\varphi}
\def\vvphi{\tilde{\vphi}}
\def\Ind{{\rm Ind}}
\providecommand{\norm}[1]{\lVert#1\rVert}
\def\eps{\varepsilon}
\def\sub{\subseteq}

\newcommand{\eq}{^{\mathrm{eq}}}
\newcommand{\cl}[2][]{\overline{#2}^{#1}}

\begin{document}

\title[]{Free actions of free groups on countable structures and property (T)}

\authors{

\author{David M. Evans}

\address{%
Department of Mathematics\\
Imperial College London\\
London SW7~2AZ\\
UK.}

\email{david.evans@imperial.ac.uk}

%

\author{Todor Tsankov}

\address{%
Universit\'e Paris 7\\
UFR de Math\'ematiques, Case 7012\\
75205 Paris \textsc{cedex} 13\\
France.}

\email{todor@math.univ-paris-diderot.fr}

}

\subjclass[2010]{Primary 22A25, 20B27; Secondary 03C15}
\thanks{The second author was partially supported by the ANR grant GrupoLoco (ANR-11-JS01-008).}

\date{}

\begin{abstract}  We show that if $G$ is a non-archimedean, Roelcke precompact, Polish group, then $G$ has Kazhdan's property (T). Moreover, if $G$ has a smallest open subgroup of finite index, then $G$ has a finite Kazhdan set. Examples of such $G$ include automorphism groups of countable $\omega$-categorical structures, that is, the closed, oligomorphic permutation groups on a countable set. The proof uses work of the second author on the unitary representations of such groups, together with a separation result for infinite permutation groups. The latter allows the construction of a non-abelian free subgroup of $G$ acting freely in all infinite transitive permutation representations of $G$.
\end{abstract}

\maketitle

\section{Introduction}

\subsection{Main results}
A topological group $G$ is \textit{non-archimedean} if it has a base of open neighbourhoods of the identity consisting of subgroups. The symmetric group $\Sym(X)$ on a set $X$, consisting of the group of all permutations of $X$, equipped with the topology of pointwise convergence, is an example of such a group: pointwise stabilizers of finite sets form a base of open neighbourhoods of the identity. It is well-known that a Polish group $G$ is non-archimedean if and only if it is isomorphic to a closed  subgroup of $\Sym(X)$ for some countable $X$. Moreover, such groups are exactly automorphism groups of first-order structures on $X$.

A group $G \leq \Sym(X)$ is said to be \textit{oligomorphic} (in its action on $X$) if $G$ has only finitely many orbits on $X^n$, for all $n \in \N$ (where the action on $X^n$ is the diagonal action). Such  groups have been extensively studied from the point of view of infinite permutation groups, combinatorics, model theory, and topological dynamics (see, for example, the references \cite{Cameron1990}, \cite{Macpherson2011a} and \cite{Kechris2014}). They arise as automorphism groups of $\omega$-categorical structures and model-theoretic methods produce a wide variety of examples of these. 

In \cite{Tsankov2012}, the second author studied the unitary representations of oligomorphic permutation groups, showing that they are completely reducible and giving a description of the irreducible representations. Many of the results of \cite{Tsankov2012} hold under a weaker (and more intrinsic) assumption than that of being an oligomorphic group: that of \textit{Roelcke precompactness} (see Definition \ref{15} here). For $G \leq \Sym(X)$ this means that whenever $Y$ is a union of finitely many $G$-orbits, then $G$ acts oligomorphically on $Y$ (see Lemma~\ref{16}). We note that in the interesting cases, Roelcke precompact groups are not locally compact: more precisely, a Roelcke precompact topological group is locally compact iff it is compact.

Using this description of the unitary representations, the paper \cite{Tsankov2012} shows that Kazhdan's property (T) holds for a natural class of closed, oligomorphic permutation groups $G \leq \Sym(X)$ (\cite{Tsankov2012}*{Theorem~6.6}; for the definition of property (T), see Definition 4.2 here). Furthermore, \cite{Tsankov2012}*{Theorem~6.7} gives some examples --- including the automorphism groups of the rational ordering $(\Q, \leq)$ and the random graph ---  of such groups with \textit{strong} property (T), where the Kazhdan set can be taken to be finite. In the latter case, the proof proceeds by finding a non-abelian free subgroup of $G$ which acts freely on $X$, an idea that goes back to Bekka~\cite{Bekka2003}. We use a similar method here, combined with techniques from permutation group theory, to prove the following general result.

\begin{theorem}\label{11} Suppose that $G$ is a non-archimedean, Roelcke precompact, Polish group and  $G^\circ$ is the intersection of the open subgroups of finite index in $G$. Then $G$ and $G^\circ$ have Kazhdan's property (T) and $G^\circ$ has a finite Kazhdan set. 
\end{theorem}

While so far property (T) has found most of its applications in the realm of locally compact groups, we note that there are some interesting consequences in our setting as well. Combining Theorem~\ref{11} with the results of Glasner and Weiss~\cite{Glasner1997}, we obtain the following.
\begin{corollary} \label{c:Bauer}
  Let $G$ be a non-archimedean, Roelcke precompact group and $G \curvearrowright X$ a continuous action on a compact Hausdorff space $X$. Then the simplex of $G$-invariant measures on $X$ is a \emph{Bauer simplex}, i.e., the set of its extreme points is closed.
\end{corollary}
The extreme points of the simplex of invariant measures are exactly the ergodic measures (a measure $\mu$ is \emph{ergodic} if every $\mu$-invariant measurable set $A \sub X$ is null or co-null; a set $A$ is \emph{$\mu$-invariant} if for every $g \in G$, $\mu(A \,\triangle\, gA) = 0$). While Glasner and Weiss only state their theorem for locally compact groups, the proof works equally well in general; see the book of Phelps~\cite{Phelps2001} for the general version of the ergodic decomposition theorem needed in the proof.

We also note that while amenable locally compact groups with property (T) must be compact, this is not necessarily true in our setting. Indeed, there are a number of automorphism groups of $\omega$-categorical structures that are amenable: for example, this is true if the structure has the so-called \emph{Hrushovski property}, see \cite{Kechris2007a} for more details. Amenability is relevant for Corollary~\ref{c:Bauer} because it ensures that the simplex of invariant measures is non-empty for any $G$-flow $X$.

Apart from the description of the representations of non-archimedean, Roelcke precompact, Polish groups given in \cite{Tsankov2012} (see Theorem~\ref{12} here), the main ingredient in the proof is the following, which is the main contribution of the current paper.

\begin{theorem} \label{13} Suppose $G$ is a non-archimedean, Roelcke precompact Polish group and $G^\circ$ is the intersection of the open subgroups of finite index in $G$. Suppose $G^\circ \neq \{1_G\}$. Then there exist $f,g \in G^\circ$ which generate a (non-abelian) free subgroup $F$ of $G^\circ$ with the property that if $H \leq G$ is open and of infinite index, then $F$ acts freely on the coset space $G/H$.
\end{theorem}

Here, recall that a group $F$ acting on a set $Y$ is acting \textit{freely} if for all non-identity $g \in F$ and $y \in Y$ we have $gy \neq y$. Thus each $F$-orbit on $Y$ is regular.

Theorems \ref{11} and \ref{13} answer Questions (1) (for non-archimedean groups) and (2) at the end of \cite{Tsankov2012}.

Recall that an action of a group $G$ on a discrete space $X$ is called \emph{amenable} if there is a $G$-invariant finitely additive measure on $X$. As a further corollary to Theorem~\ref{13}, we note:
\begin{corollary} \label{14} Suppose $G$ is a non-archimedean, Roelcke precompact Polish group which is acting continuously on the discrete space $X$. Suppose the action of $G$ on $X$ is amenable. Then $X$ contains a finite orbit.
\end{corollary}

The proof of Theorem~\ref{13} rests ultimately on Neumann's Lemma (see Lemma~\ref{21}), a very general result about separating finite sets in an infinite permutation group. 

The required  consequences of this for closed, Roelcke precompact permutation groups are given in Section~2. These results can be deduced from `folklore' results in model theory, but we provide proofs of them in the language of permutation groups. Theorem~\ref{13} is proved in Section 3 along with Corollary~\ref{14}. Section~4 discusses Kazhdan's property (T) and contains the proof of Theorem~\ref{11}. 

\subsection*{Notation} Our notation for permutation groups is fairly standard. Groups act on the left. If $G$ is a group acting on  $X$ and $A \subseteq X$, then $G_{A}$ is the pointwise stabilizer $\{g \in G : ga = a \mbox{ for all } a \in A\}$. If $A = \{a\}$ is a singleton, we denote this by $G_a$. If $G$ is the automorphism group of some structure $M$ with domain $X$, then we write $G = \Aut(M)$ and use the alternative notation $\Aut(M/A)$ for $G_{A}$. We do not usually distinguish notationally between a structure and its underlying set. 

\subsection*{Acknowledgements} Both Authors thank Dugald Macpherson for helpful discussions about some of the material in this paper. The paper was completed while the Authors were participating in the trimester programme `Universality and Homogeneity' at the Hausdorff Institute for Mathematics, Bonn. We are also grateful to the anonymous referee for carefully reading the paper and making useful suggestions.

\subsection{Background}

\begin{definition} \label{15} The topological group $G$ is called \textit{Roelcke precompact} if for every open neighbourhood $U$ of the identity, there is a finite set $E$ such that $G = UEU$.
\end{definition}

 If $G$ is non-archimedean, so $G \leq \Sym(X)$ for some $X$, then $U$ in Definition~\ref{15} can be taken to be an open subgroup and the condition for Roelcke precompactness says that there are only finitely many double cosets of $U$ in $G$. In fact, if we take $\a \in X^n$ and $U = G_{\a}$ its stabilizer, then the condition says that there are finitely many $G$-orbits on $Y^2$, where $Y$ is the $G$-orbit on $X^n$ containing $\a$. Recall that a group $G$ is said to act \textit{oligomorphically} on a set $Y$ if $G$ has finitely many orbits on $Y^n$ for all $n \in \N$. The following is a restatement of \cite{Tsankov2012}*{Theorem~2.4}:

 \begin{lemma}\label{16} Suppose $G \leq \Sym(X)$. Then $G$ is Roelcke precompact if and only if whenever $Y$ is a union of finitely many $G$-orbits on $X$, then $G$ acts oligomorphically on $Y$.
 \end{lemma}
 
 If $G$ is acting on $X$ as in the above, then we say that $G$ is \textit{locally oligomorphic} on $X$. Note that in this case, if $A \subseteq X$ is finite, then its pointwise stabilizer $G_{A}$ is also locally oligomorphic on $X$ (this follows easily from Roelcke precompactness). 
 
 We summarise the above as:
 
 \begin{corollary} A topological group $G$ is non-archimedean, Polish and Roelcke precompact if and only if it can be represented as a closed, locally oligomorphic subgroup of $\Sym(X)$, for countable $X$. 
 \end{corollary}

 Closed subgroups of $\Sym(X)$ are precisely automorphism groups of first-order structures on $X$.
Indeed, if $G \leq \Sym(X)$, we consider the \textit{canonical structure} which has a relation for each $G$-orbit on $X^n$, for all $n \in \N$. Then $G$ is a closed subgroup of $\Sym(X)$ if and only if it is the full automorphism group of this canonical structure. If $X$ is countable, then, by the Ryll-Nardzewski Theorem, $G$ is oligomorphic if and only if this canonical structure is $\omega$-categorical. (The book \cite{Cameron1990} is a convenient reference for this material.)
 
 The following fact is an easy consequence of Roelcke precompactness; see \cite{Tsankov2012}*{Corollary~2.5} for a proof.
 \begin{lemma} Suppose $X$ is countable and $G \leq \Sym(X)$ is locally oligomorphic. Then $G$ has only countably many open subgroups.
 \end{lemma}

We also note that for such $G$, there is a `universal' choice for the set $X$.

\begin{lemma}\label{19} Suppose $G$ is a non-archimedean Polish group with countably many open subgroups. Then there is a faithful action of $G$ on a countable set $X = X(G)$ with the property that for every open subgroup $U \leq G$, there is $a \in X$ such that $U = G_a$. Moreover, $G$ is closed in $\Sym(X)$ in this action.
\end{lemma}
\begin{proof} Let $(U_i : i\in I)$ be a system of representatives for the set of conjugacy classes of open subgroups of $G$. Let $X(G)$ be the disjoint union of the (left) coset spaces $G/U_i$. So $X(G)$ is countable and if $U \leq G$ is open there is a unique $i \in I$ such that $U$ is conjugate to $U_i$, so $U = gU_i g^{-1}$ for some $g \in G$. Then $U$ is the stabilizer of the coset $a = gU_i \in X(G)$. 

  It remains to prove that $G$ is a closed subgroup of $\Sym(X(G))$. As $G$ is Polish, it suffices to show that the original topology of $G$ is the same as the one inherited from $\Sym(X(G))$. Indeed, if $\{g_n\}$ is a sequence in $G$ converging to $1$ in $\Sym(X(G))$, then it eventually enters every open subgroup of $G$, which, as $G$ is non-archimedean, means that $g_n \to 1$ in $G$.

\end{proof}

\begin{remarks}
  It is worth noting that if $M$ is a countable $\omega$-categorical structure and $G = \Aut(M)$, then the action of $G$ on $X(G)$ (as in Lemma~\ref{19}) is essentially that of $G$ on $M\eq$.
\end{remarks}

We finally observe that our setting is slightly more general than the classical one of oligomorphic groups (and $\omega$-categorical structures). Locally oligomorphic groups can be represented as inverse limits of oligomorphic ones (which, however, are not necessarily closed in $\Sym(X)$). Examples can be obtained by taking the disjoint union of countably many $\omega$-categorical structures (in disjoint languages); then the automorphism group of this structure is the direct product of the automorphism groups of the individual structures and is locally oligomorphic. Perhaps more interestingly, consider the abelian group $M$ which is a direct sum of countably many copies of the Pr\"ufer $p$-group $Z(p^{\infty})$ (complex roots of $1$ of order a power of $p$) for some prime $p$. It can be checked that $\Aut(M)$ acts locally oligomorphically on $M$. Another example is given by the $\aleph_0$-partite random graph (with named parts).
Further interesting structures can be constructed from inverse limits of finite covers of $\omega$-categorical structures.

\section{Algebraic closure and Neumann's Lemma} 

The following result is sometimes called \textit{Neumann's Lemma} (cf. Corollary~4.2.2 of \cite{Hodges1993}). It is equivalent to a well known result of B.~H.~Neumann on covering groups by cosets; an independent, combinatorial proof can be found in \cite{Birch1976}, or \cite{Cameron1990}*{2.16}. 

\begin{lemma} \label{21} Suppose $G$ is a group acting on a set $X$ and all $G$-orbits on $X$ are infinite. Suppose $A, B \subseteq X$ are finite. Then there is some $g \in G$ with $gA \cap B = \emptyset$.
\end{lemma}

The result has the following model-theoretic consequence, which can be regarded as `folklore'.

\begin{lemma}\label{22} Suppose $M$ is a countable, saturated first-order structure and $A, B, C \subseteq M$ are algebraic closures of some finite subsets of $M$ with $B \subseteq C$. Then there is $g \in \Aut(M/B)$ such that $g(C) \cap A = B \cap A$. 
\end{lemma}

Here, the algebraic closure in $M$  of a set $E$ is the union of the finite $E$-definable subsets of $M$. Saturation means that if $E$ is finite and $S$ is a family of $E$-definable subsets of $M$ with the finite intersection property, then $\bigcap S \neq \emptyset$. It implies that the algebraic closure of a finite $E \subseteq M$ is the union of the finite $\Aut(M/E)$-orbits (where $\Aut(M/E)$ denotes the pointwise stabilizer of $E$ in the automorphism group $\Aut(M)$). 

We give an analogous result for  locally oligomorphic permutation groups.

\begin{definition} Suppose $G \leq \Sym(X)$ is locally oligomorphic on $X$. If $E \subseteq X$ is finite, the \textit{algebraic closure} $\acl(E)$ of $E$ in $X$ is the union of the finite $G_{E}$-orbits. 
\end{definition}
 Note that, with this notation, if $Y$ is a $G$-orbit on $X$, then $\acl(E) \cap Y$ is finite. 
 
\begin{lemma} \label{24} Suppose $M$ is countable and $G \leq \Sym(M)$ is closed and locally oligomorphic on $M$. If $E \subseteq M$ is finite, then $G_{\acl(E)}$  is also locally oligomorphic on $M$ and has no finite orbits on $M \setminus \acl(E)$ .
\end{lemma}

\begin{proof} 
Let $B = \acl(E)$. Note that $B$ is the union of a chain $E = E_0 \subseteq E_1 \subseteq E_2 \subseteq \cdots $ of finite $G_{E}$-invariant sets. 

\smallskip

\noindent\textit{Claim 1:\/} Suppose $\a \in M^n$ is a finite tuple and let $A_i$ be the $G_{E_i}$-orbit containing $\a$. Then there is $N \in \N$ such that $A_i = A_N$ for all $i \geq N$. 

Indeed, note that  $G_{E_i}$ is a normal subgroup of finite index in $G_{E}$. It follows that $\{ gA_i : g \in G_{E}\}$ is a $G_{E}$-invariant partition of $A_0$ (with finitely many parts). As $G_{E}$ has finitely many orbits on $A_0^2$, there are only finitely many possibilities for such a partition, so as $A_i \supseteq A_{i+1}$, they must be equal for sufficiently large $i$. $(\Box_{Claim\,1}.)$
\smallskip

It is worth noting that the $N$ here depends only on the $G_{E}$-orbit containing $\a$, not the particular representative $\a$. As $G_{B} = \bigcap_i G_{E_i}$, it follows that in each $G_{E}$-orbit, the $G_{B}$-orbits coincide with the $G_{E_N}$-orbits, and as $G_{E_N}$ has finite index in $G_{E}$, there are only finitely many of them, showing that $G_{B}$ is locally oligomorphic and that $G_{B} \cdot \a$ is infinite whenever $G_{E} \cdot \a$ is.

If $a \in M \setminus \acl(E)$, then by the Claim, $G_{\acl(E)} \cdot a = G_{E_N} \cdot a$ for some $N$; as $G_{E_N}$ has finite index in $G_E$ and $G_E \cdot a$ is infinite, this means that $G_{\acl(E)} \cdot a$ is also infinite.
\end{proof}

\begin{lemma}\label{25} Suppose $M$ is countable and $G \leq \Sym(M)$ is closed and locally oligomorphic on $M$. Suppose $A, B, C$ are algebraic closures in $M$ of some finite subsets of $M$ and $B \subseteq C$. Then there is $g \in G_{B}$ such that $g(C) \cap A = B \cap A$.
\end{lemma}

\begin{proof} Let $\a, \c$ be finite tuples with $A = \acl(\a)$ and $C = \acl(\c)$ (where algebraic closure is from the action of $G$ on $M$, of course). Consider the action of $H = G_{B}$ on $M$. By Lemma~\ref{24}, $H$ is locally oligomorphic on $M$ and has no finite orbits on $M \setminus B$. Let $S$ be the $H$-orbit containing $\c$ and let $S_1,\ldots, S_k$ be the $H_{\a}$-orbits on $S$ (note that there are finitely many of these, as $H$ is oligomorphic on the union of the $H$-orbits which contain the elements of the tuples $\a,\c$).

Write $M\setminus B$ as the union of a chain $X_0 \subseteq X_1 \subseteq X_2 \subseteq \cdots $ of subsets each of which is a finite union of $H$-orbits. Recall that the intersection of $A, C$ with each $X_i$ is finite. So by Lemma~\ref{21}, for each $j \in \N$, there is some $\c' \in S$ such that 
\[ (X_j \cap \acl(\c')) \cap (X_j \cap \acl(\a)) = \emptyset.\]
Thus, for each $j \in \N$ there is some $i \leq k$ such that this holds for all $\c' \in S_i$. 
So there is some $i \leq k$ such that this holds for all $j$ and for all $\c' \in S_i$. In particular, there is $\c' \in S$ such that $\acl(\c') \cap \acl(\a) \subseteq B$, as required.
\end{proof}

\begin{remarks} The result can be derived from the model-theoretic statement Lemma~\ref{22}, though some care is required as the canonical structure for a locally oligomorphic $G \leq \Sym(X)$  is not saturated if there are infinitely many $G$-orbits on $X$. However, if we regard it as a multi-sorted structure (with a sort for each of the $G$-orbits on $X$), then it is saturated (as a multi-sorted structure) and Lemma~\ref{22} also holds in this context.  Nevertheless, it seems worthwhile to offer a direct proof which does not use the model-theoretic terminology.
\end{remarks}

\begin{definition} Suppose $G$ is a topological group. Then $G^\circ$ denotes the intersection of the open subgroups of finite index in $G$.
\end{definition}

\begin{lemma}\label{28} Suppose $G$ is a non-archimedean, Roelcke precompact Polish group. Consider $G$ as a closed subgroup of $\Sym(X(G))$ (as in Lemma~\ref{19}). Then $G^\circ = G_{\acl(\emptyset)}$ and it is Roelcke precompact.  Moreover $(G^\circ)^\circ = G^\circ$.
\end{lemma}

\begin{proof} By definition, an element of $X(G)$ is in $\acl(\emptyset)$ if and only if its stabilizer is open and of finite index. Moreover, any such subgroup is the stabilizer of some point of $X(G)$, so the first statement is immediate from Lemma~\ref{24}. 

  For the second statement, suppose $U \subseteq G^\circ$ is open and of finite index.  The topology on $G^\circ$ is the subspace topology so there is some $e \in X(G)$ such that $G_e \cap G^\circ \leq U \leq G^\circ$. Let $D$ be the $U$-orbit on $X(G)$ containing $e$; let $C$ be the $G^\circ$-orbit and $E$ the $G$-orbit. Note that by Claim 1 in the proof of Lemma~\ref{24}, $C$ is equal to the $G_{X}$-orbit containing $e$, for some finite, $G$-invariant $X \subseteq \acl(\emptyset)$. Thus $\{ gC : g \in G\}$ is a finite partition of $E$. So as $U$ is of finite index in $G^\circ$, we have that $\{ gD : g \in G\}$ is also a finite partition of $E$. Let $V$ be the setwise stabilizer of $D$ in $G$. So this is an open subgroup of finite index in $G$ and $V \cap G^\circ = U$. (To see this, let $v \in V \cap G^\circ$. Then there is $u \in U$ such that $v \cdot e = u \cdot e$, so $u^{-1}v \in G_e$. Therefore $v \in uG_e \cap G^\circ = u(G_e \cap G^\circ) \sub U$.) But $V \geq G^\circ$, so $U = G^\circ$.
\end{proof}

\begin{lemma}
  \label{l:hom-compact}
  Let $G$ be a non-archimedean, Roelcke precompact, Polish group and $\pi \colon G \to K$ be a continuous homomorphism to a compact Polish group. Then $\pi(G)$ is closed in $K$. In particular, $G/G^\circ$ is a compact, profinite group.
\end{lemma}
\begin{proof}
  Without loss of generality, we may assume that $\pi(G)$ is dense in $K$. Let $V \leq G$ be an open subgroup. As $G$ is Roelcke precompact, there is a finite $F \sub G$ such that $VFV = G$. Using that $K$ is compact, we obtain
  \begin{equation*}
    K = \cl{\pi(V) \pi(F) \pi(V)} = \cl{\pi(V)} \pi(F) \cl{\pi(V)}.
  \end{equation*}
  So $K$ is the disjoint union of finitely many double cosets of the compact group $\cl{\pi(V)}$, which are closed and, therefore, open. In particular, $\cl{\pi(V)}$ is open.

  The open subgroup $V \leq G$ was arbitrary, so for all $V$, $\pi(V)$ is somewhere dense. As $G$ is Polish, this implies that $\pi(G)$ is non-meager in $K$ (see, for example, \cite{Kechris2007a}*{Proposition~3.2}). As $\pi(G)$ is also a Borel subgroup of $K$, it must be open and closed \cite{Kechris1995}*{9.11}, implying that $\pi$ is surjective.

  For the last statement, let $K = \varprojlim G/N$, where the inverse limit is taken over all finite index, open, normal subgroups of $G$ directed by reverse inclusion. (Note that, as $G$ is Roelcke precompact, it has only countably many open subgroups, so $K$ is Polish.) Then there is a natural injective homomorphism $\pi \colon G/G^\circ \to K$ with dense image. By the main statement of the lemma, $\pi$ is also surjective, and, therefore, a topological group isomorphism.
\end{proof}

\section{Constructing automorphisms without fixed points} 


\begin{theorem} \label{29} Suppose $M$ is a countable set and $G \leq \Sym(M)$ is closed and locally oligomorphic. Suppose $G^\circ \neq 1$. Then there exist elements $f, g \in G^\circ$  generating a non-abelian free subgroup $F$ of $G^\circ$ which acts freely on $M \setminus \acl(\emptyset)$. 
\end{theorem}

Suppose that $M$ and $G$ are as in the theorem. Note that by Lemma~\ref{24}, $G^\circ$ is closed and locally oligomorphic on $M\setminus \acl(\emptyset)$. Also, by Lemma~\ref{28}, $(G^\circ)^\circ = G^\circ$. So for the rest of the proof we may assume without loss of generality that $G = G^\circ$ and $\acl(\emptyset) = \emptyset$. Note that the assumption that $G^\circ \neq 1$ means that some $G^\circ$-orbit is infinite. 

We will regard $M$ as a first order structure with automorphism group $G$ (for example, by giving it its canonical structure). 

Let $\A = \{ \acl(E) : E \subseteq M \mbox{ finite}\}$ be the set of algebraic closures of finite subsets of $M$.
By a \textit{partial automorphism} of $M$ we mean a bijection $P \to Q$ between elements of $\A$ which extends to an element of $G$. We build $f,  g$ in the theorem by a back-and-forth argument as the union of a chain of partial automorphisms, $\vphi \colon A \to B$ and $\gamma \colon A' \to B'$ (approximating $f, g$ respectively). It will suffice to show how to extend the domain of one of $\vphi, \gamma$ in the `forth' step (by symmetry, the argument for extending images in the `back' step will be the same).

Consider $F_2 = \langle a, b\rangle$, the free group on generators $a,b$. The non-identity elements of $F_2$ can be thought of as reduced words $\omega(a,b)$ in $a,b$: so expressions $\omega(a,b) = c_1c_2\cdots c_r$ where $c_i \in \{a,b, a^{-1} , b^{-1}\}$ and  $c_i \neq c_{i+1}^{-1}$ for all $i < r$. If $\vphi,\gamma$ are partial automorphisms of $M$ then by $\omega(\vphi, \gamma)$ we mean the composition obtained by substituting $\vphi$ for $a$ and $\gamma$ for $b$ in $\omega(a,b)$. We refer to this as a  \textit{reduced word} in $\vphi, \gamma$. This is a bijection between subsets  of $M$ which extends to an element of $G$  (of course, it could be empty); a fixed point of this is an element $x \in M$ such that $\omega(\vphi, \gamma)x $ is defined and equal to $x$. 

Theorem~\ref{29} follows from the following:

\begin{proposition} \label{210} Suppose $M$ is as above and $\vphi \colon A \to B$, $\gamma \colon A' \to B'$ are partial automorphisms of $M$ such that no reduced word in $\vphi, \gamma$ has a fixed point. Suppose $A \subseteq C \in \A$. Then there is an extension $\vvphi : C \to D$ of $\vphi$ to a partial automorphism with domain $C$  such that no reduced word in $\vvphi, \gamma$ has a fixed point.

\end{proposition}

\begin{proof} First we show that we can choose $D$ (and $\vvphi$) so that $D \cap (C \cup B\cup A'\cup B') = B$.  As $\vphi$ extends to an element of $G$, there is some partial automorphism $\vphi' : C \to D'$ extending $\vphi$. So $B \subseteq D'$ and by Lemma \ref{25} there is $g \in G_B$ with $gD' \cap (C\cup B \cup A'\cup B') = B$. Let $D = gD'$ and $\vvphi : C \to D$ be the composition $g\circ \vphi'$. This has the required property.

Now we show that this choice of $\vvphi$ works.

\medskip

We first note the following: 

\medskip

\textit{Observation:\/} If $\vvphi^{-1}z$ is defined and $z \in C\cup B \cup A' \cup B'$, then $z \in (C\cup B \cup A' \cup B')\cap D = B$. So $z \in B$ and therefore $\vphi^{-1}z  $ is defined and $\vvphi^{-1}z = \vphi^{-1}z \in A$. Similarly, if $\vvphi w$ is defined and  $\vvphi w \in C \cup B \cup A' \cup B'$ then $\vvphi w \in (C \cup B \cup A' \cup B') \cap D = B$, so $w \in A$, $\vphi w$ is defined and $\vphi w = \vvphi w$. 

\medskip

Now suppose $\pi_1\pi_2 \cdots \pi_r$ is a reduced word in $\vvphi, \gamma$ and $x,y \in M$ are such that
\begin{equation}
  \label{eq:star}
  \pi_1\pi_2 \cdots \pi_r x = y. \tag{$\ast$}
\end{equation}
So $\pi_i \in \{ \vvphi, \vvphi^{-1}, \gamma, \gamma^{-1}\}$ and $\pi_{i+1} \neq \pi_i^{-1}$. We show that most of the terms $\vvphi$, $\vvphi^{-1}$ in this equation can be replaced by $\vphi$, $\vphi^{-1}$ without changing the validity of the equation.

\medskip

\textit{Claim:\/} Suppose $\pi_i = \vvphi^{-1}$ and $i < r$. Then we can replace $\pi_i$ by $\pi_i' = \vphi^{-1}$ in $(*)$. Similarly, if $\pi_j = \vvphi$ and $j > 1$ then we can replace $\pi_j$ by $\pi_j' = \vphi$ in \eqref{eq:star}.

To see this, note that as the word is reduced, $\pi_{i+1}$ is equal to $\gamma, \gamma^{-1}$, or $\vvphi^{-1}$ (as we will want to repeat this argument we also consider the possibility that it is $\vphi^{-1}$). If $z = \pi_{i+1}\ldots\pi_r x$ then $z$ is in the image of $\pi_{i+1}$, so $z \in A'\cup B' \cup C$. By the observation, it follows that $\vphi^{-1}z$ is defined (and equal to $\vvphi^{-1}z$), so we can replace $\pi_i$ by $\vphi^{-1}$ as required. Similarly $\pi_{j-1}$ is equal to $\gamma$, $\gamma^{-1}$ or $\vvphi$ (or $\vphi$). So $w = \pi_{j+1}\ldots \pi_r x$ is such that $\vvphi w$ is defined and in the set $A'\cup B' \cup C$. So by the observation, $\vphi w$ is defined (and equal to $\vvphi w$) and we can make the required replacement. ($\Box_{Claim}$)

\medskip

Now make all of the replacements allowed by the Claim. The only possible $\vvphi^{-1}$ remaining is if $\pi_r = \vvphi^{-1}$ and the only possible $\vvphi$ remaining is if $\pi_1 = \vvphi$.

Thus, after the replacements we have 
\begin{equation}
  \label{eq:starstar}
  \vvphi^s \beta \vvphi^{-t} x = y, \tag{$\ast\ast$}
\end{equation}
where $s, t \in \{0,1\}$ and $\beta$ is a (possibly trivial) reduced word in $\vphi, \gamma$. Suppose, for a contradiction, that $x = y$.

We consider various cases. If $\beta$ is trivial, then exactly one of $s, t$ is $1$ and (possibly after rearranging \eqref{eq:starstar}) we have $\vvphi x = x$. Then $x \in C \cap D$, so $x \in B$ and $\vphi x$ is defined. Thus $\vphi x = x$, contradicting the assumption on $\vphi, \gamma$. Suppose now that $\beta$ is non-trivial. If $s = t = 1$, then we can rearrange \eqref{eq:starstar} to obtain $\beta z = z$ where $z = \vvphi^{-1} x$. As $\beta$ is a non-trivial reduced word in $\vphi, \gamma$, this is a contradiction. We also have a contradiction if $s=t=0$. For the remaining cases, by rearranging \eqref{eq:starstar} if necessary, we can assume $s=1$ and $t=0$, that is, $\vvphi\beta x = x$. Then $x \in D \cap (A\cup B\cup A'\cup B')$, so $x \in B$, $\beta x \in A$ and $\vphi(\beta x)$ is defined, with $\vphi\beta x = x$. But $\vphi\beta$ is a non-trivial reduced word in $\gamma, \vphi$ (it comes from the same word as $\pi_1\ldots\pi_r$) so we have a contradiction.
\end{proof}


\medskip

\begin{proof}[Proof of Theorem~\ref{29}]
 Recall that we are assuming (without loss of generality) that $\acl(\emptyset) = \emptyset$ and $M$ is infinite (the latter from the assumption that $G^\circ \neq 1$).  We build chains of partial automorphisms
\[ \vphi_1 \subseteq \vphi_2 \subseteq \vphi_3 \subseteq \cdots \mbox{ and } \gamma_1 \subseteq \gamma_2 \subseteq \gamma_3 \subseteq \cdots \]
such that $f = \bigcup_i \vphi_i$ and $g = \bigcup_i \gamma_i$ are automorphisms. At each stage we use  Proposition \ref{210}  to extend the domain or image of one of $\vphi_i, \gamma_i$ so that $f, g$ will be automorphisms. We can start off with $\vphi_1, \gamma_1$ so that no reduced word in $\vphi_1, \gamma_1$ has any fixed points. To do this, we just ensure (using Lemma~\ref{25}) that the domains and images of $\phi_1, \gamma_1$ are all disjoint.  Then by the Proposition, the same will be true of all the $\vphi_i, \gamma_i$, and therefore of $f,g$. So no reduced word in $f, g$ has any fixed points: in particular, every reduced word in $f, g$ is not the identity so $f,g$ freely generate a free group whose non-identity elements have no fixed points on $M$.
\end{proof}

\medskip

We can now prove Theorem~\ref{13} from the introduction.

\medskip

\noindent\textit{Proof of Theorem~\ref{13}:\/} Consider $G$ as acting as a closed, locally oligomorphic subgroup of $\Sym(X(G))$, where $X(G)$ is as in Lemma~\ref{19}. Let $f, g$ be as given by Theorem~\ref{29} with $M = X(G)$. Suppose $H \leq G$ is open and of infinite index and $1 \neq k \in F$. By construction of $X(G)$, there is an injective $G$-morphism from the coset space $G/H$ to $X(G)$. As $H$ is of infinite index in $G$, clearly the image of this is in $M\setminus \acl(\emptyset)$. It follows that $k$ has no fixed points on $G/H$, as required. \hfill $\Box$

\medskip

We now obtain Corollary~\ref{14} from the introduction. 

\medskip

\begin{proof}[Proof of Corollary~\ref{14}]
  Suppose $G$ is as in the statement of the Corollary. If $G^\circ = 1$, then by Lemmas~\ref{24} and \ref{28}, every orbit of $G$ is finite. Otherwise, by Theorem~\ref{13} there is a non-abelian free subgroup $F$ of $G$ as in Theorem~\ref{13}. If all orbits of $X$ are infinite, then $F$ acts freely on $X$. But this is impossible if the action of $G$ (and therefore of $F$) on $X$ is amenable.
\end{proof}

\medskip

As a further application of Theorem~\ref{13}, we note the following.

\begin{corollary} Suppose $G$ is a non-archimedean, Roelcke precompact Polish group. Then $G$ is not equal to the union of its open subgroups of infinite index.
\end{corollary}

\begin{proof} Open subgroups of infinite index are exactly the stabilizers of  elements of $X(G)\setminus \acl(\emptyset)$. So the statement follows once we know that some element of $G$ fixes no  element of $X(G)\setminus \acl(\emptyset)$. But we just showed that there is a free group of rank 2 with this property, which is more than enough.
\end{proof}

Note that very little is used in the proof of Theorem~\ref{29} apart from Neumann's Lemma, in the form of Lemma~\ref{25}. As the corresponding result holds for countable, saturated structures (Lemma~\ref{22}) or, more generally, countable multi-sorted structures which are saturated as multi-sorted structures, the proof of Theorem~\ref{29} also gives the following.

\begin{corollary} \label{maincor} Suppose $M$ is a countable saturated structure with some infinite sort. Then there exist $f, g \in \Aut(M/\acl(\emptyset))$ such that $F = \langle f, g\rangle$ is the free group on $f, g$ and every non-identity element of $F$ fixes no elements of $M\eq \setminus \acl\eq(\emptyset)$.\hfill$\Box$
\end{corollary}

\begin{example} We give an example of a countable (non-saturated) structure with a rich automorphism group for which the above corollary fails. Let $L$ be a language with countably many binary relation symbols $(E_i : i \in \N)$. Consider the class $\C$ of finite $L$-structures $A$ where each $E_i$ is an equivalence relation on $A$ and only finitely many of the $E_i$ are not the universal relation $A^2$ on $A$. Then $\C$ has countably many isomorphism types and it is easy to check that it is a Fra\"{\i}ss\'e amalgamation class. Let $M$ be the Fra\"{\i}ss\'e limit. Then $M$ is a countable, homogeneous $L$-structure. 

As $M$ is constructed as the union of a chain of finite structures in $\C$, if $a,b \in M$ then there exists $n$ such that $E_n(a,b)$. In particular, $M$ is not saturated. If $g \in \Aut(M)$ and $a \in M$, then $E_n(a, ga)$ for some $n$, therefore $g$ fixes the $E_n$-class which contains $a$. It is easy to see that the $E_n$-classes are non-algebraic elements of $M\eq$, therefore every automorphism of $M$ fixes a non-algebraic element of $M\eq$. In particular, $\Aut(M)$ is the union of proper open subgroups of infinite index. 

Note that, of course, Neumann's Lemma fails for $M\eq$ in this example: $\acl\eq(a)$ contains  the $E_n$-equivalence classes of $a$ and we have just observed that there is no automorphism which moves this to a set disjoint from it (over $\acl\eq(\emptyset)$).
\end{example}

\section{Property (T) for non-archimedean, Polish, Roelcke precompact groups}

The book \cite{Bekka2008} is a convenient reference for the background to this section.

\medskip

A \textit{unitary representation} of a topological group $G$ is a homomorphism $\pi \colon G \to U(\mathcal{H})$ to the unitary group of some Hilbert space $\mathcal{H}$ which is strongly continuous, meaning that for every $\xi \in \mathcal{H}$ the map $G \to \mathcal{H}$ given by $g \mapsto \pi(g)\xi$ is continuous. If $H$ is an open subgroup of $G$, consider the action of $G$ on the coset space $Y = G/H$. Then $Y$ is a discrete space and the action of $G$ on $Y$ is continuous. It is easy to check that the corresponding action of $G$ on $\ell^2(Y)$ gives a unitary representation of $G$ (called the \textit{quasi-regular representation} $\lambda_{G/H}$).

From \cite{Tsankov2012}, we have the following:

\begin{theorem} \label{12} Suppose $G$ is a non-archimedean, Roelcke precompact, Polish group. Then every unitary representation of $G$  is a direct sum of irreducible unitary representations. Moreover, every irreducible unitary representation of $G$ is a subrepresentation of $\ell^2(G/H)$, for some open subgroup $H$ of $G$.
\end{theorem}

\begin{proof} The first statement is part of \cite{Tsankov2012}*{Theorem~4.2}. The rest of the proof is similar to that of \cite{Tsankov2012}*{Proposition 6.2}. If $\pi : G \to U(\mathcal{H})$ is an irreducible unitary representation of $G$, then by \cite{Tsankov2012}*{Theorem~4.2}, $\pi$ is isomorphic to an induced representation $\Ind_K^G(\sigma)$ for some open subgroup $K$ of $G$ and irreducible representation $\sigma$ of $K$ which factors through a finite quotient $K/H$ of $K$. In particular, $\pi$ is a subrepresentation of $\Ind_K^G(\Ind_H^K(1_H))$ (where $1_H$ is the trivial representation of $H$), which is the same thing as $\ell^2(Y)$ with $Y = G/H$.
\end{proof}

We now recall the definition of property (T) for topological groups.

\begin{definition} \label{42} Suppose $G$ is a topological group, $Q\subseteq G$ and $\eps > 0$. If $\pi \colon G \to U(\mathcal{H})$ is a unitary representation of $G$, we say that a non-zero vector $\xi \in \mathcal{H}$ is \textit{$(Q,\eps)$-invariant} (for $\pi$) if 
$\sup_{x \in Q}\norm{\pi(x)\xi - \xi} < \eps \norm{\xi}$.  

We say that $(Q,\eps)$ is a \textit{Kazhdan pair}  if for every unitary representation $\pi$ of $G$, if $\pi$ has a $(Q,\eps)$-invariant vector, then it has a (non-zero) invariant vector. We say that $G$ has \textit{Kazhdan's property (T)} (respectively, \textit{strong property (T)}) if there is a Kazhdan pair $(Q,\eps)$ with $Q$ compact (respectively, finite).
\end{definition}

The following fact will be useful (see Proposition~1.7.6 and Remark~1.7.9 in \cite{Bekka2008}).
\begin{lemma}
  \label{l:short-exact}
  Let $G$ be a completely metrizable group and $N \lhd G$ a closed normal subgroup. If both $N$ and $G/N$ have property (T), then so does $G$.
\end{lemma}

In \cite[Theorem~6.6]{Tsankov2012},  it was shown that automorphism groups of certain $\omega$-categorical structures (those without algebraicity and with weak elimination of imaginaries) have property (T). Furthermore, \cite[Theorem~6.7]{Tsankov2012} gave some examples of $\omega$-categorical structures $M$ whose automorphism groups $G$ have strong property (T). In the latter case, the proof proceeded by exhibiting a free action of a non-abelian free group on the structure. We use the same idea, together with Theorem~\ref{13}  to prove the second part of Theorem~\ref{11}: if $X$ is countable and $G \leq \Sym(X)$ is closed and Roelcke precompact, then $G^\circ$ has strong property (T). This generalises the results in \cite{Tsankov2012}. 

\medskip

\begin{proof}[Proof of Theorem~\ref{11}]

  Suppose $G$ is a non-archimedean, Roelcke precompact Polish group. By Lemma~\ref{l:hom-compact}, the quotient group $G/G^\circ$ is compact and therefore has property (T). In view of Lemma~\ref{l:short-exact}, to prove the theorem, it remains to show that $G^\circ$ has property (T). By Theorem~\ref{13}, there exists a set $Q = \{f_1,f_2\} \sub G^\circ$ which generates a non-abelian free subgroup $F$ of $G^\circ$ with the property that if $H$ is a proper, open subgroup of $G^\circ$, then $F$ acts freely on the coset space $G^\circ/H$. Following an argument similar to the one in \cite{Bekka2003}, we show that $(Q, \sqrt{2-\sqrt{3}})$ is a Kazhdan pair for $G^\circ$. By Theorem~\ref{12}, it suffices to show that for any proper open subgroup $H \leq G^\circ$ and all $\xi \in \ell^2(G^\circ/H)$, 
\begin{equation} \label{eq:1}
  \max_{i=1,2} \, \norm{\pi(f_i) \cdot \xi - \xi} \geq \sqrt{2-\sqrt{3}} \, \norm{\xi}.
\end{equation}
By Theorem~\ref{29}, the restriction of $\pi$ to $F$ is a direct sum of copies of the left-regular representation of $F$ and Kesten's theorem \cite{Kesten1959} tells us that
  \[
  \norm{\pi(f_1) + \pi(f_1^{-1}) + \pi(f_2) + \pi(f_2^{-1})} = 2 \sqrt{3}.
  \]
  A simple calculation using the Cauchy--Schwartz inequality (see \cite{Bekka2003}*{pp. 515--516} for details) yields
  \[
  \sum_{i=1}^2 \norm{\pi(f_i) \cdot \xi - \xi}^2 \geq 4 - 2\sqrt{3},
  \]
thus proving \eqref{eq:1}.
\end{proof}

\bibliography{mybiblio}

\end{document}